\documentclass[priprent]{elsarticle}
\usepackage{graphicx, epstopdf}
\usepackage{amssymb,amsmath}
\usepackage[toc,page,title,titletoc,header]{appendix}
\usepackage{amsbsy}
\usepackage{amsthm}
\usepackage{multirow}
\usepackage{indentfirst}
\usepackage{natbib}
\usepackage{amscd}
\numberwithin{equation}{section}

\newtheorem{Theorem}{Theorem}[section]

\newtheorem{Lemma}{Lemma}[section]

\newtheorem{Assumption-Notation}[Theorem]{Assumption-Notation}

\newtheorem{Proposition}{Proposition}[section]
\newtheorem{Remark}{Remark}[section]
\newtheorem{Corollary}{Corollary}[section]

\long\def\symbolfootnote[#1]#2{\begingroup
\def\thefootnote{\fnsymbol{footnote}}\footnote[#1]{#2}\endgroup}

\journal{}

\begin{document}
\begin{frontmatter}
\title{A note on refracted L\'evy processes without positive jumps}
 \author{Jiang Zhou\corref{cor1}}
 \ead{1101110056@pku.edu.cn}
 \author{Lan Wu\corref{}}
 \ead{lwu@pku.edu.cn}
 \cortext[cor1]{Corresponding author.}

\address{School of Mathematical Sciences, Peking University, Beijing 100871, PR China}

\begin{abstract}{
For a refracted spectrally negative L\'evy process, we present some  new and fantastic formulas for its q-potential measures without killing. Unlike previous results, which are written in terms of the well-known q-scale functions, our formulas are free of the q-scale functions.
 }
\end{abstract}

\end{frontmatter}

\section{Introduction}
A refracted L\'evy process $U=(U_t)_{t\geq 0}$  has the following dynamics (see (1.1) in [4]):
\begin{equation}
U_t=X_t-\delta \int_{0}^{t}\textbf{1}_{\{U_s > b\}}ds,
\end{equation}
where $\delta, b \in \mathbb R$, $X=(X_t)_{t\geq 0}$ is a L\'evy process and $\textbf{1}_{A}$ is the indicator function of a set $A$. In [4-6], the authors have investigated the process $U$ under the assumption that $X$ is a L\'evy process with negative jumps only (i.e., $X$ is a spectrally negative  L\'evy process). In [7,8], we considered a similar process $U^*$, i.e., $U^*_t=X_t-\delta \int_{0}^{t}\textbf{1}_{\{U^*_s < b\}}ds$, where $X_t$ has jumps in both positive and negative directions.

In this paper, we are interested in the q-potential measures of $U$ without killing, i.e.,
\begin{equation}
\mathbb E_x\left[\int_{0}^{\infty}e^{-q t}\rm{\bf{1}}_{\{U_t \in dy\}}dt\right],
\end{equation}
where $q>0$ and $x,y \in \mathbb R$. If $X_t$ in (1.1) a spectrally negative L\'evy process, then formulas for (1.2), which are written in terms of q-scale functions, have been derived in [4]; see Lemma 2.2 below. However, the q-scale functions are only defined for a spectrally negative L\'evy process, which poses a serious constraint on extending the results in [4] to a L\'evy process that has both negative and positive jumps. In addition, the expressions in [4] are long and complicated.

In [9], under the assumption that $X$ in (1.1) is a jump diffusion process with jumps having rational Laplace transform,  we have used the ideas in [7,8] to derive formulas for (1.2), which are written in terms of Wiener-Hopf factors (see Lemma 2.3 below). Since any other L\'evy process can be approximated by a sequence of L\'evy processes with jumps having rational Laplace transform (see Remark 2.2 in [9] or Proposition $1$ in [1]), in [9], we have proposed a conjecture which claims that Corollary 4.1 in that paper holds for a general L\'evy process and the corresponding refracted L\'evy process $U$. If this conjecture is true, then Corollary 4.1 in [9] must also hold for a refracted L\'evy process $U$ driven by a spectrally negative L\'evy process. In other words, if $X$ in (1.1) is a spectrally negative L\'evy process, then Corollary 4.1 in [9] should reduce to Theorem 6 (iv) in [4] (i.e. Lemma 2.2 below). Confirming the above fact is the main objective of this article.

The remainder of this paper is organized as follows. In Section 2, we introduce some notations and give the main result. Then, in Sections 3 and 4, the details on the derivation of the main result are given. Finally, we conclude in Section 5.

\section{Notations and Main results}
In this article, the process $X=(X_t)_{t\geq 0}$ in (1.1) is a spectrally negative L\'evy process and the parameter $\delta > 0$. The law of $X$ starting from $x$ is denoted by $\mathbb P_x$; $\mathbb E_x$ denotes the corresponding expectation; and when $x = 0$, we write $\mathbb P$ and $\mathbb E$ to simplify the notation.
In the following, some notations are introduced, and for convenience, we use the notations  appeared in [4].

The Laplace exponent $\psi(\theta)$ of $X$ is given by
\begin{equation}
\begin{split}
\psi(\theta)=\ln (\mathbb E\left[e^{\theta X_1}\right])
&=\frac{1}{2}\sigma^2\theta^2+ \gamma \theta-\int_{(1,\infty)}(1-e^{-\theta x})\Pi(dx)\\
&-\int_{(0,1)}(1-e^{-\theta x}-\theta x)\Pi(dx), \ \ \theta \geq 0,
\end{split}
\end{equation}
where $\gamma \in \mathbb R$ and $\sigma \geq 0$; the L\'evy measure $\Pi$ has a support of $(0,\infty)$ and satisfies $\int_{(0,\infty)}(x^2\wedge 1)\Pi(dx)< \infty$.
Moreover, if $X$ has bounded variation, i.e., $\sigma=0$ and $\int_{(0,1)}x\Pi(dx)<\infty$, we can write $\psi(\theta)$ as
\begin{equation}
\psi(\theta)= d \theta -\int_{0}^{\infty}(1-e^{-\theta x})\Pi(dx),
\end{equation}
with  $d=\gamma + \int_{(0,1)} x\Pi(dx)$. As in [4], we require that $d > \delta > 0$ when the process $X$ has bounded variation.

For $q, \delta > 0$, define
\begin{equation}
\Phi(q)=\sup\{\theta \geq 0: \psi(\theta)=q\} \ \ and \ \ \varphi(q)=\sup\{\theta \geq 0: \psi(\theta)-\delta \theta =q\}.
\end{equation}
Obviously, we have $\varphi(q)>\Phi(q)$. For given $q>0$, the q-scale function $W^{(q)}(x)$ associated with the process $X$ is strictly increasing and continuous on $(0,\infty)$ with the Laplace transform
\begin{equation}
\int_{0}^{\infty}e^{-s x}W^{(q)}(x)dx=\frac{1}{\psi(s)-q}, \ \ for \ \ s >\Phi(q),
\end{equation}
and $W^{(q)}(x)=0$ for $x < 0$.

The following lemma gives some properties on the q-scale function $W^{(q)}(x)$, which is taken from Lemmas 2.3 and 3.1 in [2].

\begin{Lemma}
The function $W^{(q)}(x)$ is differentiable except at a countable number of points.
And if we define $W^{(q)}(0):=\lim_{x\downarrow 0}W^{(q)}(x)$, then
\begin{equation}
W^{(q)}(0)=\left\{\begin{array}{cc}
\frac{1}{d}, & if \ X \ has  \ bounded  \ variation,\\
0, & otherwise.
\end{array}\right.
\end{equation}
\end{Lemma}

\begin{Remark}
Since $W^{(q)}(x)$ is strictly increasing and continuous on $(0,\infty)$ and Lemma 2.1 holds,  the
integral $\int_{0}^{b}e^{-s x}W^{(q)}(dx)$ for $b > 0$ can be understood as
$\int_{0}^{b}e^{-s x}W^{(q)\prime}(x)dx$.
\end{Remark}
Due to the above remark, the application of integration by parts leads to
\begin{equation}
\begin{split}
&W^{(q)}(0)+\int_{0}^{\infty}e^{-s x}W^{(q)\prime}(x)dx=\int_{[0,\infty)}e^{-s x}W^{(q)}(dx) \\
&=\int_{0}^{\infty}s e^{-s x}W^{(q)}(x)dx=\frac{s}{\psi(s)-q}, \ \ s>\Phi(q).
\end{split}
\end{equation}
where the final equality is due to (2.4).

Similar to [4], for given $\delta>0$, we introduce the process $Y=(Y_t)_{t\geq 0}$, which is defined as $Y_t=X_t-\delta t$. In this paper, we denote by $\hat{\mathbb P}_y$ the law of $Y$ starting from $y$ and by $\hat{\mathbb E}_y$ the corresponding expectation; and when $y=0$, we write briefly $\hat{\mathbb P}$ and $\hat{\mathbb E}$.  Note that the Laplace exponent of $Y$ is given by $\psi(\theta) -\delta \theta$, i.e.,
$\psi(\theta) -\delta \theta= \ln(\hat{\mathbb E}\left[e^{\theta Y_1}\right])$. For $q>0$, the q-scale function associated with the process $Y$ is denoted by $\mathbb W^{(q)}(x)$ and satisfies
\begin{equation}
\int_{0}^{\infty}e^{-s x}\mathbb W^{(q)}(x)dx=\frac{1}{\psi(s)-\delta s -q}, \ \  for \ \ s> \varphi(q),
\end{equation}
where $\varphi(q)$ is given by (2.3).

The expression of (1.2) has  been derived in [4] and is given in Lemma 2.2.
\begin{Lemma}\{Theorem 6 (iv) in [4]\}

(i) For $x,b \in \mathbb R$, $q>0$ and $y  \geq b$,
\begin{equation}
\begin{split}
&\frac{\mathbb P_x\left(U_{e(q)} \in dy\right)}{dy}=q\frac{\mathbb E_x\left[\int_{0}^{\infty}e^{-q t}\rm{\bf{1}}_{\{U_t \in dy\}}dt\right]}{dy}=q\frac{\varphi(q)-\Phi(q)}{\delta \Phi(q)}e^{-\varphi(q)(y-b)}\\
&\times \left(\delta \Phi(q)e^{-\Phi(q)b}\int_{b}^{x}e^{\Phi(q)z}\mathbb W^{(q)}(x-z)dz+e^{\Phi(q)(x-b)}\right)-q\mathbb W^{(q)}(x-y).
\end{split}
\end{equation}

(ii) For $x,b \in \mathbb R$, $q>0$ and $y < b$,
\begin{equation}
\begin{split}
&q\frac{\mathbb E_x\left[\int_{0}^{\infty}e^{-q t}\rm{\bf{1}}_{\{U_t \in dy\}}dt\right]}{dy}=-\delta q\int_{b}^{x}\mathbb W^{(q)}(x-z)W^{(q)\prime}(z-y)dz\\
&-qW^{(q)}(x-y)+q \frac{\varphi(q)-\Phi(q)}{\Phi(q)}
e^{\varphi(q)b}\int_{b}^{\infty}e^{-\varphi(q)z}W^{(q)\prime}(z-y)dz\\
&\times \left(e^{\Phi(q)(x-b)}+\delta \Phi(q)e^{-\Phi(q)b}\int_{b}^{x}e^{\Phi(q)z}\mathbb W^{(q)}(x-z)dz\right).
\end{split}
\end{equation}
\end{Lemma}
\begin{Remark}
Since $W^{(q)}(x)=\mathbb W^{(q)}(x)=0$ for $x<0$, thus
\[
\int_{b}^{x}e^{\Phi(q)z}\mathbb W^{(q)}(x-z)dz=\int_{b}^{x}\mathbb W^{(q)}(x-z)W^{(q)\prime}(z-y)dz=0, \ \ if \ \ x<b,
\]
and $W^{(q)}(x-y)=\mathbb W^{(q)}(x-y)=0$ if $x<y$.
\end{Remark}

\begin{Remark}
Formulas (2.8) and (2.9) depend  explicitly on the q-scale functions $W^{(q)}(x)$ and $\mathbb W^{(q)}(x)$. Since scale functions are associated to a spectrally negative L\'evy process, not a general L\'evy process, these results are limited.
\end{Remark}

For $T \geq 0$, define $\underline{X}_{T}:=\inf_{0\leq t \leq T}X_t$ and  $\overline{X}_{T}:=\sup_{0\leq t \leq T}X_t$. And for the process $Y$, we have similar definitions. In addition, throughout the rest of this paper, for given $q > 0$, $e(q)$ is an exponential random variable with rate $q$ and independent of all other stochastic processes.

Based on our previous research in [9], we present another expression for (1.2) in the following Lemma 2.3, which holds at least for a refracted L\'evy process $U$ driven by a jump diffusion process with jumps having rational Laplace transform (see Corollary 4.1 in [9]).
\begin{Lemma}
(i) For $q>0$ and $y \geq b$,
\begin{equation}
\begin{split}
&\mathbb P_x\left(U_{e(q)} \in dy\right)=q \mathbb E_x\left[\int_{0}^{\infty}e^{-qt}\rm{\bf{1}}_{\{U_t \in dy\}}dt\right]\\
&=\big(F_1(0)+1\big)K_q(dy-x)+ \int_{b-x}^{y-x}F_1(dy-x-z)K_q(dz),
\end{split}
\end{equation}
where $K_q(x)$ is the convolution of the probability distribution functions of $\underline{Y}_{e(q)}$ under $\hat{\mathbb P}$ and $\overline{X}_{e(q)}$ under $\mathbb P$, i.e.,
\begin{equation}
K_q(x)=\int_{(-\infty,\min\{0,x\}]}\mathbb P\left(\overline{X}_{e(q)} \leq x-z\right)\hat{\mathbb P}\left(\underline{Y}_{e(q)} \in dz \right),
\end{equation}
and $F_1(x)$ is a continuous function on $[0, \infty)$ with the Laplace transform:
\begin{equation}
\begin{split}
&\int_{0}^{\infty} e^{-s x}F_1(x)dx
=\frac{1}{s}\left(\frac{\hat{\mathbb E}\left[e^{-s \overline{Y}_{e(q)}}\right]}
{\mathbb E\left[e^{-s \overline{X}_{e(q)}}\right]}-1\right),\ \  s > 0.
\end{split}
\end{equation}

(ii) For $q>0$ and $y < b$,
\begin{equation}
\mathbb P_x\left(U_{e(q)} \in dy\right)=\big(F_2(0)+1\big)K_q(dy-x)- \int_{y-x}^{b-x}F_2(dy-x-z)K_q(dz),
\end{equation}
where $F_2(x)$ is a continuous function on $(-\infty, 0]$ with the Laplace transform:
\begin{equation}
\begin{split}
&\int_{-\infty}^{0} e^{sz}F_2(z)dz
=\frac{1}{s}\left(\frac{\mathbb E\left[e^{s \underline{X}_{e(q)}}\right]}
{\hat{\mathbb E}\left[e^{s \underline{Y}_{e(q)}}\right]}-1\right), \ \ s > 0.
\end{split}
\end{equation}
\end{Lemma}

\begin{Remark}
The advantage of the formulas in Lemma 2.3 is that they are free of the two functions $W^{(q)}(x)$ and $\mathbb W^{(q)}(x)$.
\end{Remark}

In this paper, we want to establish the following Theorem 2.1.
\begin{Theorem}
For the case that $X$ in (1.1) is a spectrally negative L\'evy process, Lemma 2.3 also holds, i.e., Lemma 2.3 will reduce to Lemma 2.2.
\end{Theorem}

Before starting the derivation of the above result in next two sections, expressions for $F_1(x)$, $F_2(x)$ and $K_q(x)$ are given in the following proposition.

\begin{Proposition}
(i) For $x\geq 0$,
\begin{equation}
F_1(x)=\frac{\varphi(q)-\Phi(q)}{\Phi(q)}e^{-\varphi(q)x}.
\end{equation}

(ii) For $x\leq 0$,
\begin{equation}
\begin{split}
&F_2(x)=\frac{\varphi(q)-\Phi(q)}{\Phi(q)}e^{-\varphi(q)x}-\frac{\delta \varphi(q)}{\Phi(q)}
W^{(q)}(-x)\\
&- \frac{\delta \varphi(q)}{\Phi(q)} \int_{x}^{0} (\varphi(q)-\Phi(q))e^{-\varphi(q)(x-z)}W^{(q)}(-z)dz.
\end{split}
\end{equation}

(iii) For $x \geq 0$,
\begin{equation}
\begin{split}
K_q(dx)=e^{-\Phi(q)x}
\frac{q(\varphi(q)-\Phi(q))}{\varphi(q)\delta}dx,
\end{split}
\end{equation}
and for $x<0$,
\begin{equation}
\begin{split}
K_q(dx)
&=e^{-\Phi(q)x}
\frac{q(\varphi(q)-\Phi(q))}{\varphi(q)\delta}dx-\frac{q\Phi(q)}{\varphi(q)}f(x)dx, \end{split}
\end{equation}
where
\begin{equation}
\begin{split}
f(x)=
\mathbb W^{(q)}(-x)
&+(\Phi(q)-\varphi(q))\int_{x}^{0}e^{-\Phi(q)(x-z)}\mathbb W^{(q)}(-z)dz,  \ \ x<0.
\end{split}
\end{equation}
\end{Proposition}

\begin{proof}
(i) It is known that (see, e.g., (8.2) in [3])
\begin{equation}
\mathbb E\left[e^{-s\overline{X}_{e(q)}}\right]=\frac{\Phi(q)}{\Phi(q)+s} \ \ and \ \ \mathbb E\left[e^{s\underline{X}_{e(q)}}\right]=\frac{q}{\Phi(q)}\frac{\Phi(q)-s}{q-\psi(s)},  \ \ s > 0.
\end{equation}
For the process $Y$, similar results hold, i.e.,
\begin{equation}
\hat{\mathbb E}\left[e^{-s\overline{Y}_{e(q)}}\right]=\frac{\varphi(q)}{\varphi(q)+s} \ \ and \ \ \hat{\mathbb E}\left[e^{s\underline{Y}_{e(q)}}\right]=\frac{q}{\varphi(q)}
\frac{\varphi(q)-s}{q-\psi(s)+\delta s},  \ \ s > 0.
\end{equation}

It follows from (2.12), (2.20) and (2.21) that
\begin{equation}
\int_{0}^{\infty}e^{-s x}F_1(x)dx=\frac{1}{s}\left(\frac{\varphi(q)(\Phi(q)+s)}{\Phi(q)(\varphi(q)+s)}-1\right)=
\frac{\varphi(q)-\Phi(q)}{\Phi(q)(\varphi(q)+s)},
\end{equation}
which leads to (2.15).

(ii) Formulas (2.14), (2.20) and (2.21) give
\begin{equation}
\int_{-\infty}^{0}e^{sx}F_2(x)dx=\frac{\Phi(q)-\varphi(q)}{\Phi(q)(\varphi(q)-s)}+
\frac{\delta \varphi(q)(\Phi(q)-s)}{\Phi(q)(q-\psi(s))(\varphi(q)-s)}, \ \ s > 0.
\end{equation}
For $s > \varphi(q)$, the right-hand side of (2.23) can be written as
\begin{equation}
\frac{\varphi(q)-\Phi(q)}{\Phi(q)}\int_{-\infty}^{0}e^{sx}e^{-\varphi(q)x}dx-
\frac{\delta \varphi(q)}{q}\mathbb E\left[e^{s \underline{X}_{e(q)}}\right]\int_{-\infty}^{0}e^{sx}e^{-\varphi(q)x}dx.
\end{equation}
From (2.23) and (2.24), for $x < 0$, we derive
\begin{equation}
F_2(x)=\frac{\varphi(q)-\Phi(q)}{\Phi(q)}e^{-\varphi(q)x}-\frac{\delta \varphi(q)}{q}
\int_{[x,0]}e^{-\varphi(q)(x-z)}\mathbb P\left(\underline{X}_{e(q)} \in dz\right).
\end{equation}
It is well-known that (see, e.g., formula (8.20) on page 219 of [3])
\begin{equation}
\mathbb P\left(-\underline{X}_{e(q)} \in dx\right)=\frac{q}{\Phi(q)}W^{(q)}(dx)-qW^{(q)}(x)dx, \ \ x \geq 0.
\end{equation}
From (2.25), (2.26) and the application of integration by parts, we derive (2.16).

(iii) For $x\geq 0$,  from (2.11), (2.20) and (2.21), it can be derived that
\begin{equation}
\begin{split}
K_q(dx)
&=\int_{(-\infty,0]}\mathbb P\left(\overline{X}_{e(q)} \in dx-z\right)\hat{\mathbb P}\left(\underline{Y}_{e(q)} \in dz \right)\\
&=\Phi(q)e^{-\Phi(q)x}\hat{\mathbb E}\left[e^{\Phi(q)\underline{Y}_{e(q)}}\right]dx=e^{-\Phi(q)x}
\frac{q(\varphi(q)-\Phi(q))}{\varphi(q)\delta}dx,
\end{split}
\end{equation}
where we have used the fact that $\psi(\Phi(q))=q$ in the final equality.

For $x<0$, applying (2.27) and integration by parts will lead to
\begin{equation}
\begin{split}
K_q(dx)
&=\int_{(-\infty,0]}\mathbb P\left(\overline{X}_{e(q)} \in  dx-z\right)\hat{\mathbb P}\left(\underline{Y}_{e(q)} \in dz \right)\\
&-\int_{(x,0]}\mathbb P\left(\overline{X}_{e(q)} \in dx-z\right)\hat{\mathbb P}\left(\underline{Y}_{e(q)} \in dz \right)\\
&=e^{-\Phi(q)x}
\frac{q(\varphi(q)-\Phi(q))}{\varphi(q)\delta}dx-\frac{q\Phi(q)}{\varphi(q)}\mathbb W^{(q)}(-x)dx\\
&-\frac{q\Phi(q)(\Phi(q)-\varphi(q))}{\varphi(q)}\int_{x}^{0}e^{-\Phi(q)(x-z)}\mathbb W^{(q)}(-z)dz,
\end{split}
\end{equation}
where in the second equality we have used a similar result to (2.26), that is,
\begin{equation}
\hat{\mathbb P}\left(-\underline{Y}_{e(q)} \in dx\right)=\frac{q}{\varphi(q)}\mathbb W^{(q)}(dx)-q \mathbb W^{(q)}(x)dx, \ \ x \geq 0.
\end{equation}
\end{proof}

Since the expression of $F_2(x)$ in (2.16) contains the function $W^{(q)}(x)$, $F_2^{\prime}(x)$ may not be well defined for a countable number of points (see Lemma 2.1). But, note that $K_q(dx)$ is absolutely continuous with respect to the Lebesgue  measure (see (2.17) and (2.18)). Thus we can rewrite (2.10) and (2.13) as follows:
\begin{equation}
\begin{split}
&\mathbb P_x\left(U_{e(q)} \in  dy\right)=q\int_{0}^{\infty}e^{-q t}\mathbb P_x\left(U_{t} \in  dy\right)dt=\\
&\left\{\begin{array}{cc}
(F_1(0)+1)K_q(dy-x)+\int_{b-x}^{y-x}F_1^{\prime}(y-x-z)K_q(dz)dy,&y \geq b,\\
(F_2(0)+1)K_q(dy-x)- \int_{y-x}^{b-x}F_2^{\prime}(y-x-z)K_q(dz)dy, & y < b,
\end{array}\right.
\end{split}
\end{equation}
where $F_1^{\prime}(x)=\frac{\varphi(q)(\Phi(q)-\varphi(q))}{\Phi(q)}e^{-\varphi(q)x}$ for all $x>0$; and for all $x < 0$, $F_2^{\prime}(x)$ is given in the following corollary. Therefore, proving Theorem 2.1 is equivalent to show that formula (2.30) will reduce to (2.8) and (2.9).

\begin{Corollary}
For $x < 0$, the derivative of $F_2(x)$ can be written informally as
\begin{equation}
\begin{split}
F_2^{\prime}(x)
&=\frac{\delta \varphi(q)}{\Phi(q)}\left(
W^{(q)\prime}(-x)-(\varphi(q)-\Phi(q)) \int_{-\infty}^{x} e^{-\varphi(q)(x-z)}W^{(q)\prime}(-z)dz\right).
\end{split}
\end{equation}
\end{Corollary}

\begin{proof}
It follows from (2.16) that
\begin{equation}
\begin{split}
F_2^{\prime}(x)
&=\frac{\delta \varphi(q)}{\Phi(q)}
W^{(q)\prime}(-x)+\frac{(\varphi(q)-\Phi(q))\varphi(q)}{\Phi(q)}\left(\delta W^{(q)}(-x)- e^{-\varphi(q)x}\right)\\
&+\frac{\delta (\varphi(q))^2}{\Phi(q)} \int_{x}^{0} (\varphi(q)-\Phi(q))e^{-\varphi(q)(x-z)}W^{(q)}(-z)dz.
\end{split}
\end{equation}
Applying integration by parts and formula (2.4)  gives
\begin{equation}
\begin{split}
&\int_{x}^{0}e^{-\varphi(q)(x-z)}W^{(q)}(-z)dz\\
&=\int_{-\infty}^{0}e^{-\varphi(q)(x-z)}W^{(q)}(-z)dz-
\int_{-\infty}^{x}e^{-\varphi(q)(x-z)}W^{(q)}(-z)dz\\
&=\frac{e^{-\varphi(q)x}}{\psi(\varphi(q))-q}-\frac{1}{\varphi(q)}\left(W^{(q)}(-x)+
\int_{-\infty}^{x}e^{-\varphi(q)(x-z)}W^{(q)\prime}(-z)dz\right),
\end{split}
\end{equation}
where in the final equality we have used $\lim_{x \uparrow \infty} e^{-\varphi(q)x}W^{(q)}(x)=0$, which is due to (2.4) and the fact that $\varphi(q)> \Phi(q)$ (see (2.3)). Then (2.31) is derived from (2.32) and (2.33).
\end{proof}

\section{The case of $y \geq b$}
In this section, we assume that $y\geq b$ is given and our objective is to prove that formula (2.30) with $y \geq b$ will reduce to (2.8).

It follows from (2.15) that $F_1(0)+1=\frac{\varphi(q)}{\Phi(q)}$. Therefore, for $y \geq b$, the right-hand side of (2.30) can be written as
\begin{equation}
\frac{\varphi(q)}{\Phi(q)}K_q(dy -x )+
\int_{b-x}^{y-x}F_1^{\prime}(y-x-z)K_q(dz)dy.
\end{equation}

For given $y \geq b$ and all $x \in \mathbb R$, we have
\begin{equation}
\begin{split}
&\int_{b-x}^{y-x}(\varphi(q)-\Phi(q))e^{-\varphi(q)(y-x-z)}e^{-\Phi(q)z}dz
=e^{\Phi(q)(x-y)}-e^{\varphi(q)(b-y)}
e^{\Phi(q)(x-b)}.
\end{split}
\end{equation}

(i) For $x < b \leq y$, formulas (2.15), (2.17) and  (3.2) produce
\begin{equation}
\begin{split}
&\frac{\varphi(q)}{\Phi(q)} K_q(dy -x )+
\int_{b-x}^{y-x}F_1^{\prime}(y-x-z)K_q(dz)dy\\
&=\frac{(\varphi(q)-\Phi(q)) q}{\delta \Phi(q)}e^{\Phi(q)(x-b)}e^{\varphi(q)(b-y)}dy.
\end{split}
\end{equation}

(ii) For $y \geq  x \geq b$, it holds that
\begin{equation}
\begin{split}
\int_{b-x}^{y-x}F_1^{\prime}(y-x-z)
&K_q(dz)
=\int_{b-x}^{0}q(\varphi(q)-\Phi(q))e^{-\varphi(q)(y-x-z)}f(z)dz\\
&\ \ -\int_{b-x}^{y-x}\frac{q(\Phi(q)-\varphi(q))^2}
{\delta \Phi(q)}e^{-\varphi(q)(y-x-z)}e^{-\Phi(q)z}dz,
\end{split}
\end{equation}
where $f(x)$ is given by (2.19). For $s > \varphi(q)>\Phi(q)$, we can derive
\begin{equation}
\begin{split}
&\int_{-\infty}^{0}e^{st}\int_{t}^{0}e^{-\varphi(q)(t-z)}f(z)dzdt=\frac{1}{s-\varphi(q)}
\int_{-\infty}^{0}e^{sz}f(z)dz\\
&=\frac{1}{s-\varphi(q)}
\left(
\frac{1}{\psi(s)-\delta s -q}+\frac{1}{s-\Phi(q)}\frac{(\Phi(q)-\varphi(q))}{\psi(s)-\delta s -q}\right)\\
&=\frac{1}{s-\Phi(q)}\frac{1}{\psi(s)-\delta s -q}=\int_{-\infty}^{0}e^{sz}e^{-\Phi(q)z}dz\int_{-\infty}^{0}e^{sz}\mathbb W^{(q)}(-z)dz,
\end{split}
\end{equation}
which means that
\begin{equation}
\begin{split}
&\int_{x}^{0}e^{-\varphi(q)(x-z)}f(z)dz=\int_{x}^{0}e^{-\Phi(q)(x-z)}\mathbb W^{(q)}(-z)dz, \ \ x \leq 0.
\end{split}
\end{equation}

Therefore, for $b\leq x \leq y$,  from (2.17), (3.2), (3.4) and (3.6),  we obtain
\begin{small}
\begin{equation}
\begin{split}
&\frac{\varphi(q)}{\Phi(q)} K_q(dy -x )+
\int_{b-x}^{y-x}F_1^{\prime}(y-x-z)K_q(dz)dy\\
&=\frac{q(\varphi(q)-\Phi(q))}{\delta \Phi(q)}e^{\varphi(q)(b-y)}
e^{\Phi(q)(x-b)}\left(1+\delta \Phi(q)
\int_{b-x}^{0}e^{\Phi(q)z}\mathbb W^{(q)}(-z)dz\right)dy.
\end{split}
\end{equation}
\end{small}

(iii) For $x> y \geq b$, we know from (2.15) and (2.18) that
\begin{equation}
\begin{split}
\int_{b-x}^{y-x}F_1^{\prime}
(y-x
&-z)
K_q(dz)
=\int_{b-x}^{y-x}q(\varphi(q)-\Phi(q))e^{-\varphi(q)(y-x-z)}\\
&\times \left(f(z)
+\frac{(\Phi(q)-\varphi(q))}
{\delta \Phi(q)}e^{-\Phi(q)z}\right)dz.
\end{split}
\end{equation}
Since
\begin{equation}
\begin{split}
\int_{b-x}^{y-x}e^{-\varphi(q)(y-x-z)}f(z)dz
&=e^{\varphi(q)(b-y)}\int_{b-x}^{0}e^{-\varphi(q)(b-x-z)}f(z)dz\\
&-\int_{y-x}^{0}e^{-\varphi(q)(y-x-z)}f(z)dz.
\end{split}
\end{equation}
Then, from (2.18), (3.2), (3.6), (3.8) and (3.9), some simple calculations yield
\begin{small}
\begin{equation}
\begin{split}
&\frac{\varphi(q)}{\Phi(q)} K_q(dy -x )+
\int_{b-x}^{y-x}F_1^{\prime}(y-x-z)K_q(dz)dy=-q\mathbb W^{(q)}(x-y)dy\\
&+\frac{q(\varphi(q)-\Phi(q))}{\delta \Phi(q)}e^{\varphi(q)(b-y)}
e^{\Phi(q)(x-b)}\left(1+\delta \Phi(q)
\int_{b-x}^{0}e^{\Phi(q)z}\mathbb W^{(q)}(-z)dz\right)dy.
\end{split}
\end{equation}
\end{small}

From (2.8), (3.1), (3.3), (3.7) and (3.10), we obtain  the desired result.

\section{The case of $y<b$}
In this section, we  want to show that formula (2.30) with $y<b$ will reduce to (2.9).

Since $F_2(0)+1=\frac{\varphi(q)}{\Phi(q)}(1-\delta W^{(q)}(0))$ (see (2.16)), for given $y<b$ and $x \in \mathbb R$, the right-hand side of (2.30) is equal to
\begin{equation}
\frac{\varphi(q)}{\Phi(q)}(1-\delta W^{(q)}(0)) K_q(dy-x)-\int_{y-x}^{b-x}F_2^{\prime}(y-x-z)K_q(dz)dy.
\end{equation}

For given $y < b$, exchanging the order of integration gives us
\begin{small}
\begin{equation}
\begin{split}
&\int_{y-b}^{0}e^{\Phi(q)t}\int_{-\infty}^{t}e^{-\varphi(q)(t-z)}W^{(q)\prime}(-z)dzdt\\
&=\frac{1}{\Phi(q)-\varphi(q)}\left(\int_{-\infty}^{0}e^{\varphi(q)z}W^{(q)\prime}(-z)dz
-\int_{y-b}^{0}e^{\Phi(q)z}W^{(q)\prime}(-z)dz\right)\\
&-\frac{e^{\Phi(q)(y-b)}}{\Phi(q)-\varphi(q)}
\int_{-\infty}^{y-b}e^{-\varphi(q)(y-b-z)}W^{(q)\prime}(-z)dz.
\end{split}
\end{equation}
\end{small}
In addition, formula (2.6) leads to
\begin{equation}
\begin{split}
\int_{-\infty}^{0}e^{\varphi(q)z}W^{(q)\prime}(-z)dz=
\frac{1}{\delta}-W^{(q)}(0),
\end{split}
\end{equation}
which is due to the fact of $\psi(\varphi(q))-\delta \varphi(q)=q$ (see the definition of $\varphi(q)$ in (2.3)).

For given $y<b$ and all $x \in \mathbb R$, from (2.17), (2.31), (4.2) and (4.3), we obtain
\begin{equation}
\begin{split}
&\int_{y-x}^{b-x}F_2^{\prime}(y-x-z)e^{-\Phi(q)z}dz=\int_{y-b}^{0}F_2^{\prime}(t)e^{-\Phi(q)(y-x-t)}dt\\
&=
\frac{\delta \varphi(q)}{\Phi(q)}
\left(\frac{1}{\delta}-W^{(q)}(0)\right)e^{\Phi(q)(x-y)}\\
&-\frac{\delta \varphi(q)}{\Phi(q)} e^{\Phi(q)(x-b)}\int_{-\infty}^{y-b}e^{-\varphi(q)(y-b-z)}W^{(q)\prime}(-z)dz.
\end{split}
\end{equation}

(i) Assume that $x \leq y < b$.  Formulas (2.17), (2.31) and (4.4) lead to
\begin{equation}
\begin{split}
&\frac{\varphi(q)}{\Phi(q)}(1-\delta W^{(q)}(0)) K_q(dy-x)-\int_{y-x}^{b-x}F_2^{\prime}(y-x-z)K_q(dz)dy\\
&=\frac{q(\varphi(q)-\Phi(q))}{\Phi(q)} e^{\Phi(q)(x-b)}\int_{-\infty}^{y-b}e^{-\varphi(q)(y-b-z)}W^{(q)\prime}(-z)dzdy.
\end{split}
\end{equation}

(ii) For $y<x \leq b$, it follows from (2.17) and (2.18) that
\begin{equation}
\begin{split}
\int_{y-x}^{b-x}F_2^{\prime}(y-x-z)
&K_q(dz)=-\frac{q\Phi(q)}{\varphi(q)}
\int_{y-x}^{0}F_2^{\prime}(y-x-z)f(z)dz\\
&+\int_{y-x}^{b-x}F_2^{\prime}(y-x-z)\frac{q(\varphi(q)-
\Phi(q))}{\delta \varphi(q)}e^{-\Phi(q)z}dz.
\end{split}
\end{equation}

For $s>0$, formula (2.14) gives
\[
F_2(0)+1-\int_{-\infty}^{0}e^{s x}F_2^{\prime}(x)dx=\frac{\mathbb E\left[e^{s \underline{X}_{e(q)}}\right]}{\hat{\mathbb E}\left[e^{s \underline{Y}_{e(q)}}\right]},
\]
which combined with (2.16), (2.19), (2.20) and (2.21), leads to
\begin{equation}
\begin{split}
&\int_{-\infty}^{0}e^{sz}\int_{z}^{0}F_2^{\prime}(z-t)f(t)dtdz=
\int_{-\infty}^{0}e^{sz}F_2^{\prime}(z)dz \int_{-\infty}^{0}e^{sz}f(z)dz\\
&=\left(\frac{\varphi(q)}{\Phi(q)}\left(1-\delta W^{(q)}(0)\right)-\frac{\mathbb E\left[e^{s \underline{X}_{e(q)}}\right]}{\hat{\mathbb E}\left[e^{s \underline{Y}_{e(q)}}\right]}\right)
\times\frac{s-\varphi(q)}{(s-\Phi(q))(\psi(s)-\delta s -q)}\\
&=\frac{\varphi(q)}{\Phi(q)}\left(1-\delta W^{(q)}(0)\right)\int_{-\infty}^{0}e^{sz}f(z)dz-\frac{\varphi(q)}{\Phi(q)}\frac{1}{\psi(s)-q}, \ \ s> \varphi(q).
\end{split}
\end{equation}
Recall formula (2.4) and note that $\varphi(q) > \Phi(q)$. Formula (4.7) yields
 \begin{equation}
\begin{split}
\int_{y-x}^{0}F_2^{\prime}(y-x-z)f(z)dz
&=\frac{\varphi(q)}{\Phi(q)}\left(1-\delta W^{(q)}(0)\right)f(y-x)\\
&-\frac{\varphi(q)}{\Phi(q)}W^{(q)}(x-y), \ \ x>y.
\end{split}
\end{equation}

From (2.18), (4.4), (4.6) and (4.8), we deduce
\begin{equation}
\begin{split}
&\frac{\varphi(q)}{\Phi(q)}(1-\delta W^{(q)}(0)) K_q(dy-x)-\int_{y-x}^{b-x}F_2^{\prime}(y-x-z)K_q(dz)dy\\
&=\frac{q(\varphi(q)-\Phi(q))}{\Phi(q)} e^{\Phi(q)(x-b)}\int_{-\infty}^{y-b}e^{-\varphi(q)(y-b-z)}W^{(q)\prime}(-z)dzdy\\
&-qW^{(q)}(x-y)dy.
\end{split}
\end{equation}

(iii) For $x>b>y$, from (2.18), it is obvious that
\begin{equation}
\begin{split}
&\int_{y-x}^{b-x}F_2^{\prime}(y-x-z)K_q(dz)\\
&=\int_{y-x}^{b-x}F_2^{\prime}(y-x-z)\left(\frac{q(\varphi(q)-
\Phi(q))}{\delta \varphi(q)}e^{-\Phi(q)z}-\frac{q\Phi(q)}{\varphi(q)}f(z)\right)dz\\
&=\int_{y-x}^{b-x}F_2^{\prime}(y-x-z)\frac{q(\varphi(q)-
\Phi(q))}{\delta \varphi(q)}e^{-\Phi(q)z}dz\\
&-\frac{q\Phi(q)}{\varphi(q)}
\int_{y-x}^{0}F_2^{\prime}(t)f(y-x-t)dt+\frac{q\Phi(q)}{\varphi(q)}
\int_{y-x}^{y-b}F_2^{\prime}(t)f(y-x-t)dt.
\end{split}
\end{equation}

We can exchange the order of integration
\[
\begin{split}
&\int_{y-x}^{y-b}\int_{-\infty}^{t}e^{-\varphi(q)(t-z)}W^{(q)\prime}(-z)dz\int_{y-x-t}^{0}
e^{-\Phi(q)(y-x-t-z)}\mathbb W^{(q)}(-z)dzdt\\
&=\int_{y-x}^{y-b}\int_{-\infty}^{t}e^{-\varphi(q)(t-z_1)}W^{(q)\prime}(-z_1)dz_1\int_{y-x-t}^{0}
e^{-\Phi(q)(y-x-t-z_2)}\mathbb W^{(q)}(-z_2)dz_2dt\\
&=\int_{b-x}^{0}\mathbb W^{(q)}(-z_2)dz_2\int_{-\infty}^{y-x-z_2}W^{(q)\prime}(-z_1)dz_1
\int_{y-x-z_2}^{y-b}e^{-\varphi(q)(t-z_1)}e^{-\Phi(q)(y-x-t-z_2)}dt\\
&+\int_{b-x}^{0}\mathbb W^{(q)}(-z_2)dz_2\int_{y-x-z_2}^{y-b}W^{(q)\prime}(-z_1)dz_1
\int_{z_1}^{y-b}e^{-\varphi(q)(t-z_1)}e^{-\Phi(q)(y-x-t-z_2)}dt,
\end{split}
\]
and then derive that
\[
\begin{split}
&\int_{y-x}^{y-b}\int_{-\infty}^{t}e^{-\varphi(q)(t-z)}W^{(q)\prime}(-z)dz\int_{y-x-t}^{0}
e^{-\Phi(q)(y-x-t-z)}\mathbb W^{(q)}(-z)dzdt\\
&=\int_{b-x}^{0}\frac{\mathbb W^{(q)}(-z_2)}{\Phi(q)-\varphi(q)} e^{-\Phi(q)(b-x-z_2)}dz_2\int_{-\infty}^{y-b}e^{-\varphi(q)(y-b-z_1)}W^{(q)\prime}(-z_1)dz_1
\\
&-\int_{b-x}^{0}\frac{\mathbb W^{(q)}(-z_2)}{\Phi(q)-\varphi(q)} e^{\varphi(q)z_2}dz_2\int_{-\infty}^{y-x-z_2}e^{-\varphi(q)(y-x-z_1)}W^{(q)\prime}(-z_1)dz_1
\\
&-\int_{b-x}^{0}\frac{\mathbb W^{(q)}(-z_2)}{\Phi(q)-\varphi(q)} e^{\Phi(q)z_2}dz_2\int_{y-x-z_2}^{y-b}e^{-\Phi(q)(y-x-z_1)}W^{(q)\prime}(-z_1)dz_1\\ &=\int_{b-x}^{0}\frac{\mathbb W^{(q)}(-z_2)}{\Phi(q)-\varphi(q)}
e^{-\Phi(q)(b-x-z_2)}dz_2\int_{-\infty}^{y-b}e^{-\varphi(q)(y-b-z_1)}W^{(q)\prime}(-z_1)dz_1
\\
&-\int_{y-x}^{y-b}\frac{\mathbb W^{(q)}(x+t-y)}{\Phi(q)-\varphi(q)} \int_{-\infty}^{t}e^{-\varphi(q)(t-z_1)}W^{(q)\prime}(-z_1)dz_1dt
\\
&-\int_{y-x}^{y-b}\int_{y-x-z_1}^{0}\frac{\mathbb W^{(q)}(-z_2)}{\Phi(q)-\varphi(q)}
e^{\Phi(q)z_2}dz_2e^{-\Phi(q)(y-x-z_1)}W^{(q)\prime}(-z_1)dz_1,
\end{split}
\]
which combined with (2.19) and (2.31), yields (after some straightforward computations)
\begin{equation}
\begin{split}
&\int_{y-x}^{y-b}F_2^{\prime}(t)f(y-x-t)dt=\frac{\delta \varphi(q)}{\Phi(q)}\int_{y-x}^{y-b}W^{(q)\prime}(-t)\mathbb W^{(q)}(x+t-y)dt\\
&+\frac{\delta \varphi(q)(\Phi(q)-\varphi(q))}{\Phi(q)} \int_{b-x}^{0}\mathbb W^{(q)}(-z_2) e^{-\Phi(q)(b-x-z_2)}dz_2\\
&\times\int_{-\infty}^{y-b}e^{-\varphi(q)(y-b-z_1)}W^{(q)\prime}(-z_1)dz_1.
\end{split}
\end{equation}

Then, from (2.18), (4.4), (4.8), (4.10) and (4.11),  we arrive at
\begin{equation}
\begin{split}
&\frac{\varphi(q)}{\Phi(q)}(1-\delta W^{(q)}(0)) K_q(dy-x)-\int_{y-x}^{b-x}F_2^{\prime}(y-x-z)K_q(dz)dy\\
&=\frac{q(\varphi(q)-\Phi(q))}{\Phi(q)} e^{\Phi(q)(x-b)}\int_{-\infty}^{y-b}e^{-\varphi(q)(y-b-z)}W^{(q)\prime}(-z)dzdy\\
&-q\left(W^{(q)}(x-y)+\delta\int_{y-x}^{y-b}W^{(q)\prime}(-t)\mathbb W^{(q)}(x+t-y)dt\right)dy\\
&+q \delta (\varphi(q)-\Phi(q))\int_{b-x}^{0}\mathbb W^{(q)}(-z) e^{-\Phi(q)(b-x-z)}dz\\
&\times\int_{-\infty}^{y-b}e^{-\varphi(q)(y-b-z)}W^{(q)\prime}(-z)dzdy.
\end{split}
\end{equation}

Therefore, the conclusion that formula (2.30) for $y<b$ will reduce to (2.9) follows from (2.9), (4.1), (4.5), (4.9) and (4.12).

\section{Conclusion}
In this paper, we confirm that Corollary 4.1 in [9] also holds for a refracted L\'evy process driven by a L\'evy process without positive jumps. This strengthens the correctness of our previous conjecture, i.e., Corollary 4.1 in [9] holds for a general L\'evy process and the associated refracted L\'evy process.
To show this conjecture, it is necessary to answer first the following question: Is there a solution $U$ to (1.1) for a general process $X$ (which has not been solved to our knowledge and is difficult to be solved). Proving the above conjecture is an interesting research problem and is one of our future research directions.

\bigskip

\end{document}